\documentclass{elsarticle}

\usepackage{geometry,t1enc,amsfonts,latexsym,amsmath,amsthm,verbatim,amssymb,graphics,mathdots,pstool,color}

\def\N{{\mathbb N}}

\def\C{{\mathbb C}}

\def\R{{\mathbb R}}
\def\Z{{\mathbb Z}}

\newcommand{\supp}{\mathsf{supp}~ }

\newcommand{\Span}{\mathsf{Span}~ }
\newcommand{\Ker}{\mathsf{Ker}~ }

\newcommand{\re}{\mathsf{Re}~ }

\newcommand{\Rank}{\mathsf{Rank}~ }

\newcommand{\dist}{\mathsf{dist} }

\newcommand{\cl}{\operatorname{\text{{\it cl}}}}
\newcommand{\interior}{\operatorname{\text{{\it int}}}}

\newcommand{\bb}[1]{\boldsymbol{#1}}

\numberwithin{equation}{section}

\newtheorem{theorem}{Theorem}[section]
\newtheorem{proposition}[theorem]{Proposition}

\newtheorem{corollary}[theorem]{Corollary}

\begin{document}

\title{On the Kronecker and Carath{\'e}odory-Fejer theorems in several variables}
\author{Fredrik Andersson, Marcus Carlsson}
\ead{fa@maths.lth.se}
\ead{mc@maths.lth.se}
\address{Centre for Mathematical Sciences, Lund University\\Box 118, SE-22100, Lund,  Sweden\\}

\begin{abstract}
Multivariate versions of the Kronecker theorem in the continuous multivariate setting has recently been published. These theorems characterize the symbols that give rise to finite rank multidimensional Hankel and Toeplitz type operators defined on general domains. In this paper we study how the additional assumption of positive semidefinite affects the characterization of the corresponding symbols, which we refer to as Carath{\'e}odory-Fejer type theorems. We show that these theorems become particularly transparent in the continuous setting, by providing elegant if and only if statements connecting the rank with sums of exponential functions. We also discuss how these objects can be discretized, giving rise to an interesting class of structured matrices that inherit these desirable properties from their continuous analogs. We describe how the continuous Kronecker theorem also applies to these structured matrices, given sufficient sampling. We also provide a new proof for the Carath{\'e}odory-Fejer theorem for block Toeplitz matrices, based on tools from tensor algebra.
\end{abstract}

\maketitle

\section{Introduction}
The connection between low-rank Hankel and Toeplitz operators and matrices, and properties of the functions that generate them play a crucial role for instance in frequency estimation \cite{actwIEEE, kung1983state,ESPRIT, MUSIC, Stoica2005}, system identification \cite{chui1991system, doyle1989state, Kung_multivariable, kung1981state} and approximation theory \cite{andersson2010nonlinear, andersson2010nonlinear_wp, andersson2011sparse, beylkin2002generalized, beylkin2005approximation, beylkin2010approximation, peter2013generalized}. The reason for this is that there is a connection between the rank of the operators, and the fact that the functions that generate these operators and matrices, respectively, ``generically'' are sums of exponentials, where the number of terms is connected to the rank of the operators and matrices (Kronecker's theorem). Moreover, adding the condition of positive semidefinite imposes further restrictions on the sums of exponentials (Carathe\'{o}dory-Fejer's theorem), a result which underlies e.g. Pisarenko's famous method for frequency estimation \cite{pisarenko}.

We provide corresponding theorems in the multidimensional setting. In contrast to the one dimensional situation, the multidimensional framework provides substantial flexibility in how to define these operators. Whereas most previous research on multidimensional Hankel and Toeplitz type operators considers ``symbols'' or ``generating sequences'' $f$ that are defined on product domains, we here consider a framework where $f$ is defined on an open connected and bounded domain $\Omega$ in $\R^d$ (satisfying mild assumptions). Besides providing beautiful new theorems, it is our hope that the new results in this paper will pave the way for applications in multidimensional frequency estimation/approximation/compression, in analogy with the  use of Toeplitz and Hankel matrices in the one dimensional setting. For this reason, we present results both in the continuous and discretized setting, and discuss how they influence each other.

To present the key ideas, we here focus mainly on the continuous theory since it is more transparent. Let us denote the multidimensional Hankel type operators by $\Gamma_f$ and their Toeplitz counterparts by $\Theta_f$ (see Section \ref{contcase} for appropriate definitions). These operators were introduced in \cite{IEOT} where it is shown that if $\Gamma_f$ or $\Theta_f$ has rank $K<\infty$, then $f$ is necessarily an exponential polynomial; \begin{equation}\label{exppolintro}
f( x)=\sum_{j=1}^J p_j(x)e^{\zeta_j\cdot x}
\end{equation}
where $J<K$, $p_j$ are polynomials in $ x=(x_1,\ldots,x_d)$, $ \zeta_j\in\C^d$ and $\zeta_j\cdot  x$ denotes the standard scalar product $${\zeta}_k\cdot {x}=\sum_{m=1}^d \zeta_{j,m}x_m.$$ Conversely, any such exponential polynomial gives rise to finite rank $\Gamma_f$ and $\Theta_f$ respectively, and there is a method to determine the rank given the ``symbol'' \eqref{exppolintro}. Most notably, the rank equals $K$ if $f$ is of the form \begin{equation}\label{expsepintro}
f( x)=\sum_{k=1}^K c_ke^{ \zeta_k\cdot x},
\end{equation}
where $c_k\in\C$ (assuming that there is no cancelation in \eqref{expsepintro}). 

\begin{figure}
\centering 	
{\includegraphics[width=0.48\textwidth]{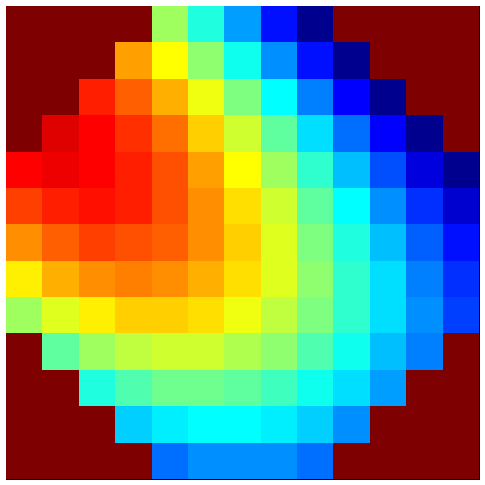}}
{\includegraphics[width=0.48\textwidth]{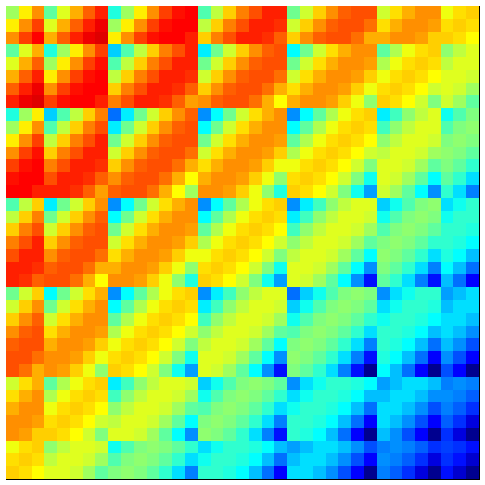}}
\caption{\label{fig:conv_support} a) ``Generating sequence'' defined on a disc $\Omega$; b) The matrix realization of the corresponding ``general domain Hankel operator'' (see Section \ref{secdisc} for further details).}
\end{figure}

In this paper, we study what happens if one adds the condition that $\Gamma_f$ or $\Theta_f$ be positive semi-definite (PSD). We prove that $\Theta_f$ then has rank $K$ if and only if $f$ is of the form \begin{equation}\label{expThetaintro}
f( x)=\sum_{k=1}^K c_ke^{i \xi_k\cdot x}
\end{equation}
where $c_k>0$ and $\xi_k\in\R^d$ (Theorem \ref{ty}), which in a certain sense is an extension of Carath{\'e}odory-Fejer's theorem on PSD Toeplitz matrices. Correspondingly, $\Gamma_f$ is PSD and has rank $K$ if and only if $f$ is of the form \begin{equation}\label{expGammaintro}
f( x)=\sum_{k=1}^K c_ke^{ \xi_k\cdot x}
\end{equation}
where again $c_k>0$ and $\xi_k\in\R^d$ (Theorem \ref{fisher}). Similar results for Hankel matrices date back to work of Fischer \cite{fischer1911}.

The only of the above results that has a simple counterpart in the finite dimensional discretized multivariable setting is the Carath{\'e}odory-Fejer's theorem, which has been observed previously in \cite{stoicachinese} (concerning block Toeplitz matrices). In this paper we provide a general result on tensor products, which can be used to ``lift'' structure results in one-dimension to the multi-dimensional setting. We use this to give an alternative proof of the discrete Carath{\'e}odory-Fejer theorem, which subsequently is used to prove the continuous counterpart.

Fischer's theorem on the other hand has no neat version in the multivariable finite dimensional setting, but has been generalized to so called small Hankel operators on $\ell^2(\N^d)$ in \cite{power}, a paper which also contains a result analog to \eqref{expGammaintro}.

However, the product domain setting is rather restrictive and not always a natural choice. Whereas one-dimensional ``symbols'' necessarily are defined on an interval, there is an abundance of possible regions to define their corresponding multidimensional cousins. Despite this, the majority of multivariate treatments of these issues are set either directly in a block-Toeplitz/Hankel setting, or rely on tensor products. In both cases the corresponding domain of definition $\Omega$ of the symbol is a square (or multi-cube), but for concrete applications to multidimensional frequency estimation, the available data need not be naturally defined on such a domain. In radially symmetric problems, a circle may be more suitable or, for boundary problems, a triangle might be more appropriate.

Concerning analogs of the above results for the discretized counterparts of $\Theta_f$ and $\Gamma_f$, we show in this paper how to construct ``structured matrices'' that approximate their continuous counterparts, and hence can be expected to inherit these desirable properties, given sufficient sampling rate. We give simple conditions on the regularity of $f$ and $\Omega$ needed for this to be successful. This gives rise to an interesting class of structured matrices, which we call ``general domain Hankel/Toeplitz matrices''. As an example, in Figure \ref{fig:conv_support} we have a ``generating sequence'' $f$ on a discretized disc, together with a plot of its general domain Hankel matrix. 

The paper is organized as follows. In the next section we review the theory and at the same time introduce the operators we will be working with in the continuous setting (Section \ref{contcase}). The short Section \ref{tensorsec} provides a tool from tensor algebra, and also introduce useful notation for the discrete setting. Section \ref{sec4} discuss how to discretize the $\Gamma_f$'s and $\Theta_f$'s, and we discuss particular cases such as block Toeplitz and Hankel matrices. In Section \ref{sec5} we prove the Caratheodory-Fejer theorem in the discrete (block) setting. Section \ref{multikronecker} shows that the discrete operators approximate the continuous counterparts, given sufficient sampling rate, and we discuss Kronecker's theorem. Sections \ref{seclast} and \ref{secPSDGAmma} considers structure results under the PSD condition, first for $\Theta_f$'s and then for $\Gamma_f$'s. In the last section, we extend the above results to the corresponding operators on unbounded domains.

\section{Review of the field}

Suppose that a Hankel matrix $H$ or a Toeplitz matrix $T$ of size $N\times N$ is taken from samples of a ``discretized exponential polynomial'' \begin{equation}\label{a}\sum_{j=1}^J p_j(n)\lambda_j^n\end{equation} (where $\lambda_j\in\C$ and $n=1,\ldots,2N-1$) of total degree \begin{equation}\label{c}K=\sum_{j=1}^J(\deg p_j +1)\end{equation} strictly less than $N$.  Based on the fundamental theorem of algebra, one can show that the rank of either $H$ or $T$ equals $K$, and that the polynomials $p_j$ and the $\lambda_j$'s are unique. This observation is sometimes used hand-wavingly in the converse direction, which is not true. However, in terms of applications this doesn't matter because of the following stronger statement: If $T$ or $H$ has rank $K<N$ then its generating sequence is ``generically'' of the form \begin{equation}\label{b}\sum_{k=1}^K c_k\lambda_k^{n}.\end{equation} In the multidimensional setting, the corresponding statement is false. We refer to \cite{IEOT} for a longer discussion of these issues, especially Section 8. See also \cite{power,gu}. As an example of an exceptional case in the one-variable situation, consider the Hankel matrix \begin{equation}\label{hankelcounter}\left(
                                                                                 \begin{array}{ccccc}
                                                                                   1 & 0 & 0 & 0 & 0 \\
                                                                                   0 & 0 & 0 & 0 & 0 \\                                                                                   0 & 0 & 0 & 0 & 0 \\                                                                                   0 & 0 & 0 & 0 & 0 \\                                                                                   0 & 0 & 0 & 0 & 1 \\                                                                                 \end{array}
                                                                               \right)
 \end{equation} Clearly, the rank is 2 but the generating sequence $(1,0,0,0,0,0,0,0,1)$ is neither of the form \eqref{a} nor \eqref{b}. The book \cite{iokhvidov}, which has two sections devoted entirely to the topic of the rank of finite Toeplitz and Hankel matrices, give a number of exact theorems relating the rank with the ``characteristic'' of the corresponding matrix, which is a set of numbers related to when determinants of certain submatrices vanish. Another viewpoint has been taken by B. Mourrain et. al \cite{mourrainSymmetric,mourrainbook,lasserre,mourrain2000}, in which, loosely speaking, these matrices are analyzed using projective algebraic geometry and the 1 to the bottom right corresponds to the point $\infty$.

In either case, the complexity of the theory does not reflect the relatively simple interaction between rank and exponential sums, as indicated in the introduction. There are however a few exceptions in the discrete setting. Kronecker's theorem says that for a Hankel operator (i.e. an infinite Hankel matrix acting on $\ell^2(\N)$), the rank is $K$ if and only if the symbol is of the desired form \eqref{a} ($0^0$ defined as 1), with the restriction that $|\lambda_j|<1$ if one is only interested in bounded operators, see e.g. \cite{chui,katayama,kronecker,peller}. Also, it is finite rank and PSD if and only if the symbol is of the form \eqref{b} with $c_k>0$ and $\lambda_k\in(-1,1)$, a result which also has been extended to the multivariable (tensor product) setting \cite{power}. In contrast, there are no finite rank \textit{bounded} Toeplitz operators (on $\ell^2(\N)$). If boundedness is not an issue, then a version of Kroneckers theorem holds for Toeplitz operators as well \cite{EllisLay}.

Adding the PSD condition for a Toeplitz matrix yields a simple result which is valid (without exceptions) for finite matrices. This is the essence of what usually is called the Carath{\'e}odory-Fejer theorem. The result was used by Pisarenko \cite{pisarenko} to construct an algorithm for frequency estimation. Since then, this approach has rendered a lot of related algorithms, for instance the MUSIC method \cite{MUSIC}. We reproduce the statement here for the convenience of the reader. For a proof see e.g. Theorem 12 in \cite{krein} or Section 4 in \cite{szego}. Other relevant references include \cite{AAKCar,Ciccariello}.

\begin{theorem}\label{tr}
Let $T$ be a finite ${N+1}\times {N+1}$ Toeplitz matrix with generating sequence $(f_n)_{n=-N}^{N}$. Then $T$ is PSD and $\Rank T=K\leq N$ if and only if $f$ is of the form \begin{equation}\label{gd}
f(n)=\sum_{k=1}^Kc_k\lambda_k^n
\end{equation}
where $c_k>0$ and the $\lambda_k$'s are distinct and satisfy $|\lambda_k|=1$.
\end{theorem} The corresponding situation for Hankel matrices $H$ is not as clean, since \eqref{hankelcounter} is PSD and has rank 2, but do not fit with the model \eqref{gd} for $c_k>0$ and real $\lambda_k$'s. Results of this type seems to go back to Fischer \cite{fischer1911}, and we will henceforth refer to statements relating the rank of PSD Hankel-type operators to the structure of their generating sequence/symbol, as ``Fischer-type theorems'' (see e.g. Theorem 5 \cite{krein} or \cite{fischer1911}). 
Corresponding results in the full rank case is found e.g. in \cite{tyrtyshnikov}.

We end this subsection with a few remarks on the practical use of Theorem \ref{tr}.
For a finitely sampled signal, the autocorrelation matrix can be estimated by $H^*H$ where $H$ is a (not necessarily square) Hankel matrix generated by the signal. This matrix will obviously be PSD, but in general it will not be Toeplitz. However, under the assumption that the $\lambda_k$'s in \eqref{gd} are well separated, the contribution from the scalar products of the different terms will be small
and might therefore be ignored. Under these assumptions on the data, the matrix $H^*H$ is PSD and approximately Toeplitz, which motivates the use of the Carath{\'e}odory-Fejer theorem as a means to retrieve the $\lambda_k$'s.

\subsection{Finite interval convolution and correlation operators}
The theory in the continuous case is much ``cleaner'' than in the discrete case. In this section we introduce the integral operator counterpart of Toeplitz and Hankel matrices, and discuss Kronecker's theorem in this setting.

Given a function on the interval $[-2,2]$, we define the finite interval convolution operator  $\Theta_f:L^2([-1,1])\rightarrow L^2([-1,1])$ by \begin{equation}\label{TCO1}\Theta_f(g)(x)=\int f(x-y)g(y)~dy.\end{equation} Replacing $x-y$ by $x+y$ we obtain the finite interval correlation operators $\Gamma_f$. These operators also go by the names Toeplitz and Hankel operators on the Paley-Wiener space. It is easy to see that if we discretize these operators, i.e. replace integrals by finite sums, then we get Toeplitz and Hankel matrices, respectively. More on this in Section \ref{secdisc}.

Kronecker's theorem (as formulated by Rochberg \cite{rochberg}) then states that $\Rank \Theta_f=K$ (and $\Rank \Gamma_f=K$) if and only if $f$ is of the form \begin{equation}\label{exp1d}f(x)=\sum_{j=1}^J p_j(x) e^{\zeta_j x}\end{equation} where $p_j$ are polynomials and $\zeta_j\in \C$. Moreover, the rank of $\Theta_f$ (or $\Gamma_f$) equals the total degree \begin{equation}\label{rank1d}K=\sum_{j=1}^J(\deg p_j+1).\end{equation}
However, functions of the form
\begin{equation}
  \label{expsep1d}f(x)=\sum_{k=1}^K c_k e^{\zeta_k x},\quad c_k,\zeta_k\in\C
\end{equation}
are known to be dense in the set of all symbols giving rise to rank $K$ finite interval convolution operators. Hence, the general form \eqref{exp1d} is hiding the following simpler statement, which often is of practical importance.
\textit{$\Theta_f$ generically  has rank $K$ if and only if $f$ is a sum of $K$ exponential functions.} As already noted, this is false in several variables, see \cite{IEOT}. The polynomial factors appear in the limit if two frequencies in \eqref{expsep1d} approach each other and interfere destructively, e.g.
\begin{equation}\label{gg}x=\lim_{\epsilon\rightarrow 0^+}\frac{e^{\epsilon x}-1}{\epsilon}.\end{equation}
This can heuristically explain why these factors do not appear when adding the PSD condition, since the functions on the right of \eqref{gg} give rise to one large positive and one large negative eigenvalue.

\subsection{The multidimensional continuous setting: TCO's}\label{contcase}

Given any (square integrable) function $f$ on an open connected and bounded set $\Omega$ in $\R^d$, $d\geq 1$, the natural counterpart to the operator
\eqref{TCO1} is the (General Domain) Truncated Convolution Operator (TCO) $\Theta_f:L^2(\Upsilon)\rightarrow L^2(\Xi)$ defined by \begin{equation}\label{TCOd}\Theta_f(g)({x})=\int_{\Upsilon}f({x}-{y})
g({y}) ~d{y},\quad {x}\in\Xi,\end{equation}
where $\Xi$ and $\Upsilon$ are connected open bounded sets such that
\begin{equation}\label{eq01}\Omega=\Xi-\Upsilon=\{{x}-{y}:~{x}\in\Xi,~{y}\in\Upsilon\}.\end{equation}
In \cite{IEOT} such TCO operators are studied, and their finite rank structure is completely characterized. It is easy to see that $\Theta_f$ has rank $K$ whenever $f$ has the form \begin{equation}\label{rep}
f({x})=\sum_{k=1}^K c_k e^{{\zeta}_k\cdot {x}},\quad {x}\in\Omega,
\end{equation}
where the ${\zeta_1},\ldots,{\zeta}_K\in\C^d$.

The reverse direction is however not as neat as in the one-dimensional case. It is true that the rank is finite only if $f$ is an exponential polynomial (Theorem 4.4 in \cite{IEOT}), but there is no counterpart to the simple formula \eqref{rank1d}. However, Section 5 (in \cite{IEOT}) gives a complete description of how to determine the rank given the symbol $f$ explicitly, Section 7 gives results on the generic rank based on the degree of the polynomials that appear in $f$, and we also provide lower bounds, and Section 8 investigates the fact that polynomial coefficients seem to appear more frequently in the multidimensional setting. Section 9 contains an investigation related to boundedness of these operators for the case of unbounded domains, which we will treat briefly in Section \ref{secUB} of the present paper.

If we instead set $\Omega=\Xi+\Upsilon$ then we may define the Truncated Correlation Operator
\begin{equation}\label{TCOdd}\Gamma_f(g)({x})=\int_{\Upsilon}f({x}+{y})
g({y}) ~d{y},\quad {x}\in\Xi.\end{equation}
This is the continuous analogue of finite Hankel (block) matrices. As in the finite dimensional case, there is no real difference
between $\Gamma_f$ and $\Theta_f$ regarding the finite rank structure. In fact, one turns into the other under composition with the ``trivial'' operator $\iota(f)({x})=f(-{x}),$ and thus all statements concerning the rank of one can easily be transferred to the other.

\subsection{Other multidimensional versions}\label{secmult}

The usual multidimensional framework is that of block-Hankel and block-Toeplitz matrices, tensor products, or so called ``small Hankel operators on $\ell^2(\N^d)$. In all cases, the generating sequence $f$ is forced to live on a product domain. For example, in \cite{hua} they consider the generating sequences of the form \eqref{expsepintro} (where $ x$ is on a discrete grid) and give conditions on the size of the block Hankel matrices under which the rank is $K$, and in \cite{stoicachinese} it is observed that the natural counterpart of the Carath{\'e}odory-Fejer theorem can be lifted by induction to the block Toeplitz setting. For the full rank case, factorizations of these kinds of operators have been investigated in \cite{feldmann, tis}. Extensions to multi-linear algebra are addressed for instance in \cite{papy2005exponential, papy2007shift,papy2009exponential}.

Concerning ``small Hankel operators'', in addition to \cite{power} we wish to mention \cite{gu} where a formula for actually determining the rank appears, although this is based on reduction over the dimension and hence not suitable for non-product domains.


There is some heuristic overlap between \cite{IEOT} and \cite{golyandina20102d,golyandina2009algebraic}. In \cite{golyandina20102d} they consider block Hankel matrices with polynomial symbols, and obtain results concerning their rank (Theorem 4.6) that overlap with Propositions 5.3, Theorem 7.4 and Proposition 7.7 of \cite{IEOT} for the 2d case. Proposition 7 in \cite{golyandina2009algebraic} is an extension to 2d of Kronecker's theorem for infinite block Hankel matrices (not truncated), which can be compared with Theorem 4.4 in \cite{IEOT}.

In the discrete setting, the work of B. Mourrain et al. considers a general domain context, and what they call ``quasi Toeplitz/Hankel matrices'' correspond to what here is called ''general domain Toeplitz/Hankel matrices'' (we stick to this term since we feel it is more informative for the purposes considered here). See e.g. Section 3.5 in \cite{mourrain2000}, where such matrices are considered for solving systems of polynomial equations. In \cite{mourrainSymmetric}, discrete multidimensional Hankel operators (not truncated) are studied, and Theorem 5.7 is a description of the rank of such an operator in terms of decompositions of related ideals. Combined with Theorem 7.34 of \cite{mourrainbook}, this result also implies that the symbol must be of the form \eqref{a}. (See also Section 3.2 of \cite{lasserre}, where similar results are presented.) These results can be thought of as a finite dimensional analogue (for product domains) of Theorem 1.2 and Proposition 1.4 in \cite{IEOT}. Theorem 5.9 gives another condition on certain ideals in order for the generating sequence to be of the simpler type, i.e. the counterpart of \eqref{expsepintro} instead of \eqref{exppolintro}. In Section 6 of the same paper they give conditions for when these results apply also to the truncated setting, based on rank preserving extension theorems. Similar results in the one-variable setting is found in Section 3 of \cite{EllisLay}.

Finally, we remark that the results in this paper concerning finite rank PSD Hankel operators partially overlap heuristically with results of \cite{power} and those found in Section 4 in \cite{lasserre}, where the formula \eqref{b} is found in the (non-truncated) discrete environment. In the latter reference they subsequently provide conditions under which this applies to the truncated setting.

With these remarks we end the review and begin to present the new results of this paper. For the sake of introducing useful notation, it is convenient to start with the result on tensor products, which will be used to ``lift'' the one-dimensional Carath{\'e}odory-Fejer theorem to the multidimensional discrete setting.

\section{A property of tensor products}\label{tensorsec}
Let $U_1,\ldots,U_d$ be finite dimensional linear subspaces of $\C^n$. Then $\otimes_{j=1}^d U_j$ is a linear subspace of $\otimes_{j=1}^d \C^n$, and the latter can be identified with the set of $\C$-valued functions on $\{1,\ldots,n\}^d$. Given $f\in\otimes_{j=1}^d \C^n$ and $\boldsymbol{x}\in\{1,\ldots,n\}^d$, we will write $f(\boldsymbol{x})$ for the corresponding value. For fixed $\boldsymbol{x}=(x_1,\ldots,x_{d-1})\in\{1,\ldots,n\}^{d-1}$ we define $$f_1(\boldsymbol{x})=f(\cdot ,x_1,\dots, x_d), f_2(\boldsymbol{x})=f(x_1,\cdot,x_2, \cdots, x_d),~ f_3(\boldsymbol{x})=f(x_1,x_2,\cdot, x_3, \cdots, x_d),$$ i.e. the vectors obtained from $f$ by freezing all but one variable. We refer to these vectors as ``probes'' of $f$. If $f\in \otimes_{j=1}^d U_j$ then it is easy to see that all probes $f_j$ of $f$ will be elements of $U_j$, $j=1,\ldots,d$. The following theorem states that the converse is also true.

\begin{theorem}\label{thmtensor}
If all possible probes $f_j(\bb x)$ of a given $f\in\otimes_{j=1}^d \C^n$ lie in $U_j$, then $f\in\otimes_{j=1}^d U_j$.
\end{theorem}
\begin{proof}
First consider the case $d=2$. Let $V\subset\otimes_{j=1}^2 \C^n$ consist of all $f$ with the property stated in the theorem. This is obviously linear and $U_1\otimes U_2\subset V$. If we do not have equality, we can pick an $f$ in $V$ which is orthogonal to $U_1\otimes U_2$. At least one probe $f_{1}(\cdot)$ must be a non-zero element $u_1$ of $U_1$. Given any $u_2\in U_2$ consider \begin{equation}\label{tensor}\langle u_1\otimes u_2, f \rangle =\sum_{j=1}^n u_2(j)\langle u_1 ,f_1(j)\rangle=\langle u_2, \sum_{j=1}^n u_1(i) f_2(i)\rangle.\end{equation}
From the middle representation and the choice of $u_1$, we see that at least one value of the vector $\sum_{j=1}^n u_1(i) f_2(i)$ is non-zero. Moreover this is a linear combination of probes $f_2(i)$, and hence an element of $U_2$. But then we can pick $u_2\in U_2$ such that the scalar product \eqref{tensor} is non-zero, which is a contradiction to the choice of $f$. The theorem is thus proved in the case $d=2$.

The general case now easily follows by induction on the dimension, noting that $\otimes_{j=1}^d \C^d$ can be identified with $\C^d\otimes(\otimes_{j=2}^{d} \C^d)$ and that $\otimes_{j=1}^d U_j$ under this identification turns into $U_1\otimes(\otimes_{j=2}^{d} U_j)$.
\end{proof}

\section{General domain Toeplitz and Hankel operators and matrices}\label{sec4}
The operators in the title arise as discretizations of general domain truncated convolution/correlation operators. These become ``summing operators'', which can be represented as matrices in many ways.

\begin{figure}
\begin{minipage}[c]{0.58\textwidth}{\includegraphics[width=\textwidth]{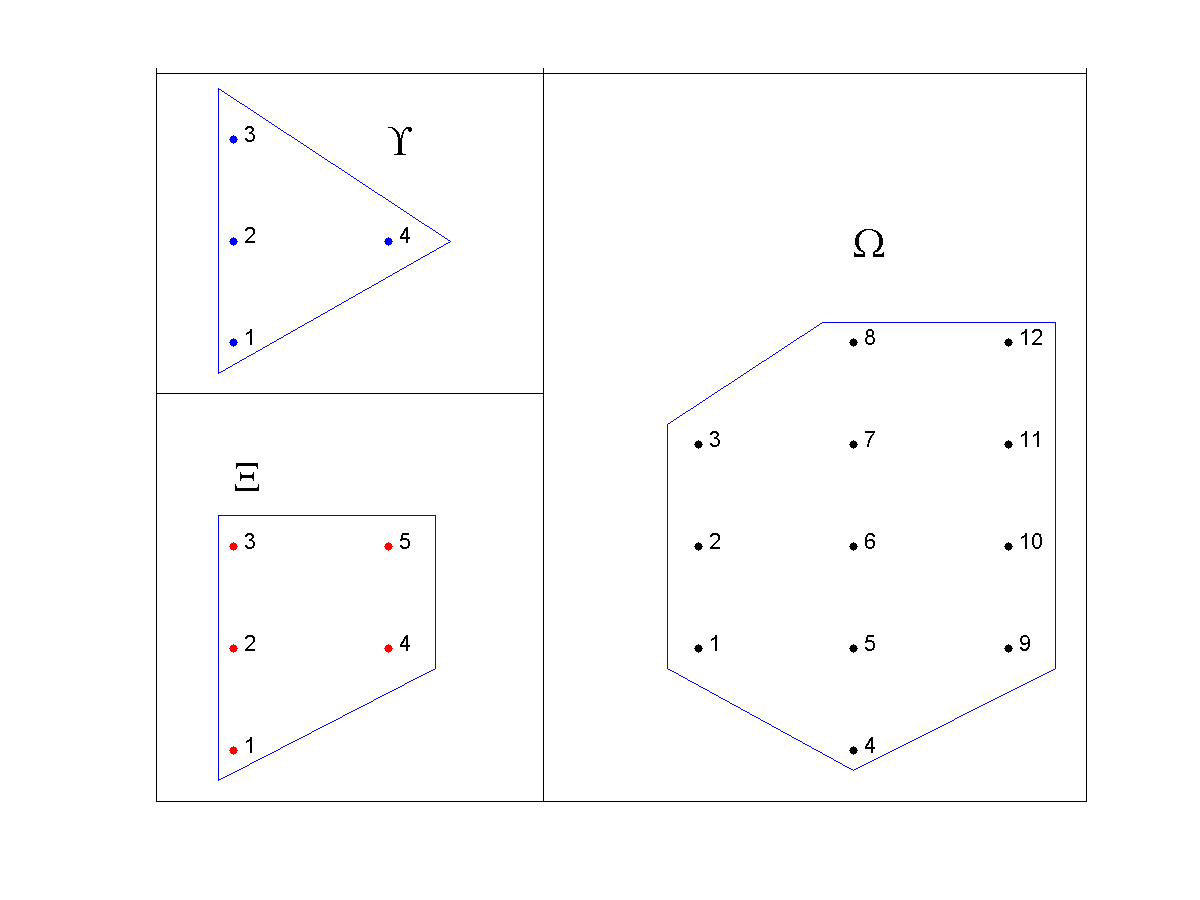}}\end{minipage}
\fbox{\begin{minipage}[c]{0.38\textwidth}
\vspace*{1.2cm}
\text{Matrix realization of $\bb {\Theta}_f$:}\[\left(
  \begin{array}{cccc}
    6 & 5 & 4 & 1 \\
    7 & 6 & 5 & 2 \\
    8 & 7 & 6 & 3 \\
    11 & 10 & 9 & 6 \\
    12 & 11 & 10 & 7 \\
  \end{array}
\right)\]
\vspace*{0.7cm}
\end{minipage}}
\caption{\label{fig:support} Left: Domains $\bb \Xi$, $\bb \Upsilon$, and $\bb\Omega=\bb \Xi-\bb \Upsilon$ with lexicographical order. Right: Illustration of where the numbered points in $\bb \Omega$ show up in the corresponding matrix realization of $\bb\Theta_f$.}
\end{figure}

\subsection{Discretization}\label{secdisc}
For simplicity of notation, we here discretize using an integer grid, since grids with other sampling lengths (these are considered in Section \ref{secdisc2}) can be obtained by first dilating the respective domains. We will throughout the paper use $\textbf{bold}$ symbols for discrete objects, and normal font for their continuous analogues. Set $$\boldsymbol{\Upsilon}=\{\boldsymbol{x}\in\Z^d:~\boldsymbol{x}\in{\Upsilon}\},$$ make analogous definition for $\Xi/\boldsymbol{\Xi}$ and define $\boldsymbol{\Omega}=\boldsymbol{\Upsilon}-\boldsymbol{\Xi}.$ We let $\boldsymbol{\Theta}_{f}$ denote what we call a ``general domain Toeplitz operator'' \begin{equation}\label{sumop}\boldsymbol{\Theta}_{f}(g)(\boldsymbol{x})=\sum_{\boldsymbol{y}\in\boldsymbol{\Upsilon}}f(\boldsymbol{x}-\boldsymbol{y})g(\boldsymbol{y}),\quad \boldsymbol{x}\in \boldsymbol{\Xi},\end{equation} where $g$ is an arbitrary function on $\boldsymbol{\Upsilon}$.
We may of course represent $g$ as a vector, by ordering the entries in some (non-unique) way. More precisely, by picking any bijection \begin{equation}\label{bij}o_y:\{1,\ldots,|\boldsymbol{\Upsilon}|\}\rightarrow\boldsymbol{\Upsilon}\end{equation}
we can identify $g$ with the vector $\tilde{g}$ given by $$(\tilde{g}_j)_{j=1}^{|\boldsymbol{\Upsilon}|}=g(o_y(j)).$$
Letting $o_x$ be an analogous bijection for $\boldsymbol{\Xi}$, it is clear that $\boldsymbol{\Theta}_{f}$ can be represented
as a matrix, where the $(i,j)$'th element is $f(o_x(i)-o_y(j))$. Such matrices will be called ``general domain Toeplitz matrices'', see Figure \ref{fig:support} for a small scale example. We make analogous definitions for $\Gamma_f$ and denote the corresponding discrete operator by $\boldsymbol{\Gamma}_f$. We refer to this as a ``general domain Hankel operator'' and its matrix realization as
``general domain Hankel matrix''. An example of such is shown in Figure \ref{fig:conv_support}.

\subsection{Block Toeplitz and Hankel matrices}\label{secblock}

If we let $\boldsymbol{\Xi}$ and $\boldsymbol{\Upsilon}$ be multi-cubes and the ordering bijections be the lexicographical order, then the matrix realization $\boldsymbol{\Theta}_f$ of \eqref{sumop} becomes a block Toeplitz matrix. These are thus special cases of the more general operators considered above. Similarly, block Hankel matrices arise when representing $\boldsymbol{\Gamma}_f$ in the same way.

For demonstration we consider $\boldsymbol{\Xi}=\boldsymbol{\Upsilon}=\{-1,0,1\}^3$ so $\boldsymbol{\Omega}=\{-2,\ldots,2\}^3$. The lexicographical order then orders $\{-1,0,1\}^3$ as $$(1,1,1),~(1,1,0),~(1,1,-1),~(1,0,1),~(1,0,0),\ldots,(-1,-1,-1).$$
The matrix-realization $T$ of a multidimensional Toeplitz operator $\boldsymbol{\Theta}_f$ then becomes

\resizebox{\linewidth}{!}{\arraycolsep=2pt%
$
T=\left(
      \begin{array}{lllllllll}
        T_{f_3(0,0)} & T_{f_3(0,-1)} & T_{f_3(0,-2)} & T_{f_3(-1,0)} & T_{f_3(-1,-1)} & T_{f_3(-1,-2)} & T_{f_3(-2,0)} & T_{f_3(-2,-1)} & T_{f_3(-2,-2)} \\
        T_{f_3(0,1)} & T_{f_3(0,0)} & T_{f_3(0,-1)} & T_{f_3(-1,1)} & T_{f_3(-1,0)} & T_{f_3(-1,-1)} & T_{f_3(-2,1)} & T_{f_3(-2,0)} & T_{f_3(-2,-1)} \\
        T_{f_3(0,2)} & T_{f_3(0,1)} & T_{f_3(0,0)} & T_{f_3(-1,2)} & T_{f_3(-1,1)} & T_{f_3(-1,0)} & T_{f_3(-2,2)} & T_{f_3(-2,1)} & T_{f_3(-2,0)} \\
        T_{f_3(1,0)} & T_{f_3(1,-1)} & T_{f_3(1,-2)} & T_{f_3(0,0)} & T_{f_3(0,-1)} & T_{f_3(0,-2)} & T_{f_3(-1,0)} & T_{f_3(-1,-1)} & T_{f_3(-1,-2)} \\
        T_{f_3(1,1)} & T_{f_3(1,0)} & T_{f_3(1,-1)} & T_{f_3(0,1)} & T_{f_3(0,0)} & T_{f_3(0,-1)} & T_{f_3(-1,1)} & T_{f_3(-1,0)} & T_{f_3(-1,-1)} \\
        T_{f_3(1,2)} & T_{f_3(1,1)} & T_{f_3(1,0)} & T_{f_3(0,2)} & T_{f_3(0,1)} & T_{f_3(0,0)} & T_{f_3(-1,2)} & T_{f_3(-1,1)} & T_{f_3(-1,0)} \\
        T_{f_3(2,0)} & T_{f_3(2,-1)} & T_{f_3(2,-2)} & T_{f_3(1,0)} & T_{f_3(1,-1)} & T_{f_3(1,-2)} & T_{f_3(0,0)} & T_{f_3(0,-1)} & T_{f_3(0,-2)} \\
        T_{f_3(2,1)} & T_{f_3(2,0)} & T_{f_3(2,-1)} & T_{f_3(1,1)} & T_{f_3(1,0)} & T_{f_3(1,-1)} & T_{f_3(0,1)} & T_{f_3(0,0)} & T_{f_3(0,-1)} \\
        T_{f_3(2,2)} & T_{f_3(2,1)} & T_{f_3(2,0)} & T_{f_3(1,2)} & T_{f_3(1,1)} & T_{f_3(1,0)} & T_{f_3(0,2)} & T_{f_3(0,1)} & T_{f_3(0,0)} \\
      \end{array}
    \right)
$
}
where e.g. $$T_{f_3(0,0)}=\left(
                            \begin{array}{ccc}
                              f(0,0,0) & f(0,0,-1) & f(0,0,-2) \\
                              f(0,0,1) & f(0,0,0) & f(0,0,-1) \\
                              f(0,0,2) & f(0,0,1) & f(0,0,0) \\
                            \end{array}
                          \right)
$$
Note that this matrix has a Toeplitz structure on 3 levels, since each $3\times 3$-block of the large matrix above is Toeplitz, and these blocks themselves form a $3\times 3$ Toeplitz structure.

\subsection{Exponential sums}\label{expsum}
We pause the general development to note some standard facts that will be needed in what follows. Fix $N\in\N$, and for $j=1,\ldots,d$ let $\Phi_j$ be a set of at most $2N$ numbers in $\C$. Set ${\Phi}=\Phi_1\times,\ldots\times\Phi_d$.
\begin{proposition}\label{linindep}
The set $\{e^{{\zeta}\cdot\boldsymbol{x}}:~{\zeta}\in {\Phi}\}$ is linearly independent as functions on $\{-N,\ldots,N\}^d$.
\end{proposition}
\begin{proof}
If $d=1$ the result is standard, see e.g. Proposition 1.1 in \cite{EllisLay} or \cite[Sec. 3.3]{brauer}. For $d>1$, the function $e^{{\zeta}\cdot\boldsymbol{x}}=e^{{\zeta}_1\bb{x}_1}\ldots e^{{\zeta}_d\bb{x}_d}$ is a tensor. The desired conclusion now follows from the $d=1$ case and standard tensor product theory.
\end{proof}

We now set $\boldsymbol{\Upsilon}=\boldsymbol{\Xi}=\{-N,\ldots,N\}^d$, and let $\boldsymbol{\Omega}=\{-2N,\ldots,2N\}^d$ in accordance with subsection \ref{secdisc}. Consider functions $f$ on $\boldsymbol{\Omega}$ given by  \begin{equation}\label{rep1}
f(\boldsymbol{x})=\sum_{k=1}^K c_k e^{{\zeta}_k\cdot \boldsymbol{x}}.
\end{equation}
We say that the representation \eqref{rep1} is reduced if all ${\zeta}_k$'s are distinct and the corresponding coefficients $c_k$ are non-zero.
\begin{proposition}\label{rankprop}
Let ${\Phi}$ be as before. Let the function $f$ on $\{-2N,\ldots,2N\}^d$ be of the reduced form \eqref{rep1} where each ${\zeta}_k$ is an element of ${\Phi}$.
 Then $$\Rank \boldsymbol{\Theta}_f=\Rank\boldsymbol{\Gamma}_f=K.$$
\end{proposition}
\begin{proof}
Pick a fixed ${\zeta}$ and consider $f(\bb x)=e^{{\zeta}\cdot \boldsymbol{x}}$ then $$\boldsymbol{\Theta}_{f}(g)(\boldsymbol{x})=\sum_{\boldsymbol{y}\in\boldsymbol{\Upsilon}}e^{{\zeta}\cdot\boldsymbol{x}}e^{-{\zeta}\cdot\boldsymbol{y}}g(\boldsymbol{y})=e^{\zeta\cdot\boldsymbol{x}}\langle g, \overline{e^{-{\zeta}\cdot\boldsymbol{y}}}\rangle,$$ which has rank 1. For a general $f$ of the form \eqref{rep1} the rank will thus be less than or equal to $K$. But Proposition \ref{linindep} implies that the set $\{e^{{\zeta}_k\cdot\boldsymbol{x}}\}_{k=1}^K$ is linearly independent as functions on $\boldsymbol{\Xi}$. Thus the rank will be precisely $K$, as desired.
\end{proof}

We end this section with a technical observation concerning 1 variable.

\begin{proposition}\label{prop1}
Let $f$ be a vector of length $m>n+1$ and $K<n$. Let $\zeta_1,\ldots,\zeta_K$ be fixed and suppose that each sub-vector of $f$ with length $n+1$ can be written of the form \eqref{rep1}, then $f$ can be written in this form as well.
\end{proposition}
\begin{proof}
Consider two adjacent sub-vectors with overlap of length $n$. On this overlap the representation \eqref{rep1} is unique, due to Proposition \ref{linindep}. The result now easily follows.
\end{proof}

\section{The multidimensional discrete Carath{\'e}odory-Fejer theorem}\label{sec5}
Throughout this section, let $\boldsymbol{\Upsilon},~\boldsymbol{\Xi}$ and $\boldsymbol{\Omega}$ be as in Sections \ref{secblock} and \ref{expsum}, i.e. multi-cubes centered at 0. The following theorem was first observed in \cite{stoicachinese}, but using a completely different proof.
\begin{theorem}\label{CFMD} Given $f$ on $\{-2N,\ldots,2N\}^d$, suppose that $\boldsymbol{\Theta}_f$ is PSD and has rank $K$ where $K\leq 2N$. Then $f$ can be written as \begin{equation}\label{expMD}f(\boldsymbol{x})=\sum_{k=1}^K c_k e^{i {\xi}_k\cdot \boldsymbol{x}}\end{equation} where $c_k>0$ and ${\xi}_k\in \R^d$ are distinct and unique. Conversely, if $f$ has this form then $\boldsymbol{\Theta}_f$ is PSD with rank $K$.
\end{theorem}
\begin{proof}
First assume that $\boldsymbol{\Theta}_f$ is PSD and has rank $K$. Let $T$ be a block Toeplitz representation of $\boldsymbol{\Theta}_f$, as described in Section \ref{secblock}. Since the rank of $T$ is $K$, we can write its singular value decomposition as \begin{equation}\label{rankrep} T=\sum_{k=1}^K s_k u_{k}  v_k^T,\end{equation} where $s_k>0$, $u_k,v_k$ are column matrices and $T$ denotes transpose. Recall from Section \ref{secblock} that the Toeplitz matrix $T_{f_d(\boldsymbol{0})}$ is the \mbox{$2N+1\times 2N+1$} sub-matrix on the diagonal of $T$, (and $\bb 0\in \R^{d-1}$). This is clearly PSD so by the classical Carath{\'e}odory-Fejer theorem (Theorem \ref{tr}), $f_d(\boldsymbol{0})$ can be represented by \begin{equation}\label{repCF1d}\sum_{k=1}^K c_k e^{i\xi_k \bb x},\quad \bb x\in \{-2N\ldots 2N\}.\end{equation} We identify functions on $\{-N\ldots N\}$ with $\C^{2N+1}$ in the obvious way, and define $U_1\subset \C^{2N+1}$ by $$U_1=\Span\{e^{i\xi_1 \bb x},\ldots,e^{i\xi_K\bb x}\}.$$ The analogous subspace of $\C^{4N+1}$ will be called $U_1^{ext}$. Note that $f_d(\boldsymbol{0})\in U_1^{ext}$ by \eqref{repCF1d}. Set $\Phi_1=\{\xi_1,\ldots,\xi_K\}$.

For $1\leq m\leq (2N+1)^{d-1}$, let $u_{k,m}$ denote the sub-vector of $u_{k}$ related to the $m$th column-block, (i.e. with subindex ranging between $(m-1)(2N+1)+1$ and $m(2N+1)$), and define $v_{k,m}$ similarly. Then \eqref{rankrep} implies that
$$T_{f_d{(\boldsymbol{0})}}=\sum_{k=1}^Ks_k u_{k,m} v_{k,m}^T.$$
But this means that each $u_{k,m}$ is in $U_1$, since $f_d(\boldsymbol{0})\in U_1^{ext}$. Fix $\boldsymbol{y}\in\{-2N,\ldots,2N\}^{d-1}$. Restricting \eqref{rankrep} to a corresponding off-diagonal Toeplitz matrices in $T$ gives, with appropriate choice of $m$ and $j$, the representation $$T_{f_d{(\boldsymbol{y})}}=\sum_{k=1}^Ks_k u_{k,m} v_{k,m+j}^T.$$  But this means that the columns of $T_{f_d{(\boldsymbol{y})}}$ lie in $U_1$. We conclude that each sub-vector of ${f_d{(\boldsymbol{y})}}$ of length $2N+1$ is in $U_1$. By Proposition \ref{prop1}, we conclude that each probe ${f_d{(\boldsymbol{y})}}$ is in $U_1^{ext}$.

By choosing a different ordering and repeating the above argument, we conclude that for each $l$ ($1\leq l\leq d$), there is a corresponding subspace $U_l^{ext}$ such that all possible probes $f_l(\cdot)$ are in $U_{l}^{ext}$. Let ${\xi}_k\in\R^d$ be an enumeration of all $K^d$ multi-frequencies $\Phi_1\times\ldots\times\Phi_d$. The corresponding $K^d$ exponential functions span $\otimes_{j=1}^d U_j$. By Theorem \ref{thmtensor} we can thus write \begin{equation}\label{ggg}f(x)=\sum_{k=1}^{K^2} c_k e^{i {\xi}_k\cdot \boldsymbol{x}}.\end{equation}
However, by Proposition \ref{rankprop}, precisely $K$ of the coefficients $c_k$ are non-zero. This is \eqref{expMD}. The uniqueness of the multi-frequencies is immediate by Proposition \ref{linindep} (applied with $N:=2N$). The linear independence of these functions also give that the coefficients are unique. To see that $c_k$ is positive, ($1\leq k\leq K$), just pick a multi-sequence on $\boldsymbol{\Xi}$ which is orthogonal to all other $e^{i {\xi}_j\cdot \boldsymbol{x}}$, $j\neq k$. Using the representation \eqref{sumop} it is easy to see that \begin{equation}\label{apa}0\leq\langle \boldsymbol{\Theta}_{f}(g),g\rangle=c_k|\langle g,e^{i {\xi}_k\cdot \boldsymbol{x}}\rangle|^2,\end{equation} and the first statement is proved.

For the converse, let $f$ be of the form \eqref{expMD}. Then $\bb{\Theta}_f$ has rank $K$ by Proposition \ref{rankprop} and the PSD property follows by the fact that \begin{equation}\label{apa1}0\leq\langle \boldsymbol{\Theta}_{f}(g),g\rangle=\sum_{k=1}^Kc_k|\langle g,e^{i {\xi}_k\cdot \boldsymbol{x}}\rangle|^2,\end{equation} in analogy with \eqref{apa}.
\end{proof}

It is possible to extend this result to more general domains as considered in Section \ref{secdisc}. However, such extensions are connected with some technical conditions, which are not needed in the continuous case. Moreover, in the next section we will show that the discretizations of Section \ref{secdisc} capture the essence of their continuous counterparts, given sufficient sampling. For these reasons we satisfy with stating such extensions for the continuous case, see Section \ref{seclast}.

The above proof could also be modified to apply to block Hankel matrices, but since Fischer's theorem is connected with preconditions to rule out exceptional cases, the result is not so neat. (It does however provide alternative proofs the results in \cite{power}, i.e. concerning small Hankel operators). We here present only the cleaner continuous version, see Section \ref{secPSDGAmma}.


\section{The multidimensional discrete Kronecker theorem}\label{multikronecker}
If we want to imitate the proof of Theorem \ref{CFMD} in Kronecker's setting, i.e. without the PSD assumption, then we have to replace \eqref{repCF1d} with \eqref{exp1d}. With suitable modifications, the whole argument goes through up until \eqref{ggg}, where now the $ \xi_k$'s can lie in $\C^d$ and $c_k$ also can be monomials. However, the key step of reducing the ($K^2$-term) representation \eqref{ggg} to the ($K$-term) representation \eqref{expMD}, via Proposition \ref{rankprop}, fails. Thus, the only conclusion we can draw is that $f$ has a representation \begin{equation}\label{rep2}
f(\boldsymbol{x})=\sum_{j=1}^J p_j(\boldsymbol{x}) e^{{\zeta}_j\cdot \boldsymbol{x}},\quad \boldsymbol{x}\in\Omega,
\end{equation}
where $J\leq K$, but we have very little information on the amount of terms in each $p_j$. This is a fundamental difference compared to before. In \cite{IEOT} examples are presented of TCO's generated by a single polynomial $p$, where $\Gamma_p$ has rank $K$ much lower than the amount of monomials needed to represent $p$. It is also not the case that these polynomials necessarily are the limit of functions of the form \eqref{rep} (in a similar way as \eqref{gg}), and hence we can not dismiss these polynomials as ``exceptional''. To obtain similar examples in the finite dimensional setting considered here, one can just discretize the corresponding TCO's found in \cite{IEOT} (as described in Section \ref{secdisc}). 

Nevertheless, in the continuous setting (i.e. for operators of the form $\Theta_f$ and $\Gamma_f$, c.f. \eqref{TCOd} and \eqref{TCOdd}) the correspondence between rank and the structure of $f$ is resolved in \cite{IEOT}. In particular it is shown that (either of) these operators have finite rank if and only if $f$ is an exponential polynomial, and that the rank equals $K$ if $f$ is of the (reduced) form \begin{equation}\label{rep3}f=\sum_{k=1}^Kc_ke^{\zeta_k\cdot  x}.\end{equation} We now show that these results apply also in the discrete setting, given that the sampling is sufficiently dense. For simplicity of notation, we only consider the case $\Gamma_f$ from now on, but include the corresponding results for $\Theta_f$ in the main theorems.

\subsection{Discretization}\label{secdisc2}

Let bounded open domains $\Upsilon,~\Xi$ be given, and let $l>0$ be a sampling length parameter. Set $$\boldsymbol{\Upsilon}_l=\{\boldsymbol{n}l\in\Upsilon:~\boldsymbol{n}\in\Z^d\},$$ (c.f. \eqref{sumop}), make analogous definition for $\boldsymbol{\Xi}_l$ and define $\boldsymbol{\Omega}_l=\boldsymbol{\Upsilon}_l+\boldsymbol{\Xi}_l.$ We denote the cardinality of $\boldsymbol{\Upsilon}_l$ by $|\boldsymbol{\Upsilon}_l|$, and we define $\ell^2(\boldsymbol{\Upsilon}_l)$ as the Hilbert space of all functions $g$ on $\boldsymbol{\Upsilon}_l$ and norm $$\|g\|_{\ell^2}=\sum_{\boldsymbol{y}\in\boldsymbol{\Upsilon}_l}|g({\boldsymbol{y}})|^2 .$$ We let $\boldsymbol{\Gamma}_{f,l}:\ell^2(\boldsymbol{\Upsilon}_l)\rightarrow\ell^2(\boldsymbol{\Xi}_l)$ denote the summing operator $$\boldsymbol{\Gamma}_{f,l}(g)(\boldsymbol{x})=\sum_{\boldsymbol{y}\in\boldsymbol{\Upsilon}_l}f(\boldsymbol{x}+\boldsymbol{y})g(\boldsymbol{y}),\quad \boldsymbol{x}\in \boldsymbol{\Xi}_l.$$
When $l$ is understood from the context, we will usually omit it from the notation to simplify the presentation. It clearly does not matter if $f$ is defined on $\Xi+\Upsilon$ or $\bb \Xi_l+\bb \Upsilon_l$, and we use the same notation in both cases. We define $\boldsymbol{\Theta}_{f,l}$ in the obvious analogous manner. Note that in Section \ref{sec4} and \ref{sec5} we worked with $\boldsymbol{\Theta}_{f}$, which with the new notation becomes the same as $\boldsymbol{\Theta}_{f,1}$.

\begin{proposition}\label{p1}
There exists a constant $C>0$ such that $$\|\boldsymbol{\Gamma}_{f,l}\|\leq Cl^{-d/2}\|f\|_{\ell^2(\boldsymbol{\Omega}_l)}.$$
\end{proposition}
\begin{proof}
By the Cauchy-Schwartz inequality we clearly have $$|\boldsymbol{\Gamma}_{f,l}(g)(\boldsymbol{x})|\leq\|f_{\boldsymbol{\Omega}_l}\|_{\ell^2(\boldsymbol{\Omega}_l)}\|g\|_{\ell^2(\boldsymbol{\Upsilon}_l)}$$
for each $\boldsymbol{x}\in\boldsymbol{\Xi}_l$. If we let $|\boldsymbol{\Xi}_l|$ denote the amount of elements in this set, it follows that
$$\|\boldsymbol{\Gamma}_{f,l}\|\leq\|f_{\boldsymbol{\Omega}_l}\|_{\ell^2(\boldsymbol{\Omega}_l)}|\boldsymbol{\Xi}_l|^{1/2}.$$
Since $\Xi$ is a bounded set, it is clear that $|\boldsymbol{\Xi}_l|l^d$ is bounded by some constant, and hence the result follows.
\end{proof}

\begin{theorem}\label{t1}
Let $f\in L^2(\Omega)$ be continuous. Then $$\Rank \boldsymbol{\Gamma}_{f,l}\leq \Rank\Gamma_f$$ (and $\Rank \boldsymbol{\Theta}_{f,l}\leq \Rank\Theta_f$).
\end{theorem}
\begin{proof}
Given ${\boldsymbol{y}}\in\boldsymbol{\Upsilon}_l$ and $t\leq l$ let $C_{\boldsymbol{y}}^{l,t}$ denote the multi-cube with center ${\boldsymbol{y}}$ and sidelength $t$, i.e. $C_{\boldsymbol{y}}^{l,t}=\{y\in \R^d:~|y-\bb y|_{\infty}<t/2\}$, where $|\cdot|_{\infty}$ denotes the supremum norm in $\R^d$. Chose $t_0$ such that $\sqrt{d}t_0/2<\dist(\boldsymbol{\Upsilon}_l,\partial\Upsilon)$. For $t<t_0$ we then have that the set $\{e^{l,t}_{\boldsymbol{y}}\}_{{\boldsymbol{y}}\in\boldsymbol{\Upsilon}_l}$ defined by $e^{l,t}_{\boldsymbol{y}}=t^{-d/2}\textbf{1}_{C_{\boldsymbol{y}}^{l,t}}$ is orthonormal in $L^2(\Upsilon)$. We make analogous definitions for $\boldsymbol{\Xi}_l$.
Clearly $\ell^2({\boldsymbol{\Upsilon}_l})$ is in bijective correspondence with $\Span\{e^{l,t}_{\boldsymbol{y}}\}_{{\boldsymbol{y}}\in\boldsymbol{\Upsilon}_l}$ via the canonical map $P^{l,t}$, i.e. $P^{l,t}(\delta_{\boldsymbol{y}})=e^{l,t}_{\boldsymbol{y}}$ where $\delta_{\boldsymbol{y}}$ is the ``Kronecker $\delta-$function''. Let $Q^{l,t}$ denote the corresponding map $Q^{l,t}:\ell^2({\boldsymbol{\Xi}_l})\rightarrow L^2(\Xi)$.

Now, clearly $\Rank {Q^{l,t}}^*\Gamma_f P^{l,t}\leq \Rank\Gamma_f$ and $$\frac{1}{t^d}\langle {Q^{l,t}}^*\Gamma_f P^{l,t}\delta_{\boldsymbol{y}},\delta_{\boldsymbol{x}}\rangle=\frac{1}{t^{2d}}\int_{|x-\bb{x}|_\infty <t/2}\int_{|y-\bb{y}|_\infty <t/2}f(x+y) ~dy~dx.$$ If we denote this number by $\tilde{f}^t({\bb{x}+\bb{y}})$, we see that $\frac{1}{t^{d}}{Q^{l,t}}^*\Gamma_f P^{l,t}=\bb{\Gamma}_{\tilde{f}^t,l}$. It follows that $\Rank \bb{\Gamma}_{\tilde{f}^t,l}\leq \Rank\Gamma_f$. Since $f$ is continuous, it is easy to see that $\lim_{t\rightarrow 0^+}\tilde{f}^t({\bb{x}+\bb{y}})=f(\bb{x}+\bb{y})$, which implies that $\lim_{t\rightarrow 0^+}\boldsymbol{\Gamma}_{\tilde{ f}^t,l}=\boldsymbol{\Gamma}_{f,l}$, and the proof is complete.
\end{proof}

\subsection{From discrete to continuous}

Our next result says that for sufficiently small $l$, the inequality in Theorem \ref{t1} is actually an equality. This needs some preparation. Given ${\boldsymbol{y}}\in\boldsymbol{\Upsilon}_l$ we now define $C_{\boldsymbol{y}}^{l}$ to be the multi-cube with center ${\boldsymbol{y}}$ and sidelength $l$. Set $\bb\Upsilon^{int}_l=\{\bb y\in\bb\Upsilon_l: ~C_{\bb y}\subset \Upsilon\}$, i.e. the set of those $\bb y$'s whose corresponding multicubes are not intersecting the boundary. Moreover, for each $\bb y \in \bb \Upsilon_l$, set
$$e^{l}_{\boldsymbol{y}}=\left\{\begin{array}{ll}
l^{-d/2}\textbf{1}_{C_{\boldsymbol{y}}^{l}}, & \textit{ if } \bb y\in\bb\Upsilon^{int}_l \\
0, & else
\end{array}\right.$$
We now define $P^{l}:\ell^2({\boldsymbol{\Upsilon}_l})\rightarrow L^2(\Upsilon)$ via $P^{l}(\delta_{\boldsymbol{y}})=e^{l}_{\boldsymbol{y}}$. Note that this map is only a partial isometry, in fact, ${P^{l}}^*{P^{l}}$ is the projection onto $\Span\{\delta_{\boldsymbol{y}}:~\bb y\in \bb\Upsilon_l^{int}\}$, and ${P^{l}}{P^{l}}^*$ is the projection in $L^2(\Upsilon)$ onto the corresponding subspace. We make analogous definitions for $\bb{\Xi}_l$, denoting the corresponding partial isometry by $Q^l$. Set $$N_l=N_l(\Upsilon)=|\bb \Upsilon_l\setminus\bb\Upsilon_l^{int}|,$$ i.e. $N_l$ is the amount of multi-cubes $C^l_{\bb{y}}$ intersecting the boundary of $\Upsilon$, and note that $N_l=\dim \Ker P^l$. Since $\Upsilon$ is bounded and open, it is easy to see that $|\bb\Upsilon_l^{int}|$ is proportional to $1/l^d$. We will say that the boundary of a bounded domain $\Upsilon$ is \emph{well-behaved} if \begin{equation}\label{defwellbehaved}\lim_{l\rightarrow0^+}l^d N_l =0.\end{equation} In other words, $\partial\Upsilon$ is well behaved if the amount of multi-cubes $C^l_{\bb{y}}$ properly contained in $\Upsilon$ asymptotically outnumbers the amount that are not.
\begin{proposition}\label{t5}
Let $\Upsilon$ be a bounded domain with Lipschitz boundary. Then $\partial\Upsilon$ is well behaved.
\end{proposition}

\begin{proof}
By definition, for each point $x\in\partial\Upsilon$ one can find a local coordinate system such that $\partial\Upsilon$ locally is the graph of a Lipschitz function from some bounded domain in $\R^{d-1}$ to $\R$, see e.g. \cite{verchota} or \cite{evans}, {Sec. 4.2}. It is not hard to see that each such patch of the boundary can be covered by a collection of balls of radius $l$, where the amount of such balls is bounded by some constant times $1/l^{d-1}$. Since $\partial \Upsilon$ is compact, the same statement applies to the entire boundary. However, it is also easy to see that one ball of radius $l$ can not intersect more than $3^d$ multi-cubes of the type $C_{\bb{y}}^l$, and henceforth $N_l$ is bounded by some constant times $1/l^{d-1}$ as well. The desired statement follows immediately.
\end{proof}

We remark that all bounded convex domains have well behaved boundaries, since such domains have Lipschitz boundaries, (see e.g.~ \cite[Sec. 6.3]{evans}). Also, note that the above proof yielded a faster decay of $N_l l^d$ than necessary, so most ``natural'' domains will have well-behaved boundaries. We are now ready for the main theorem of this section:
\begin{theorem}\label{t2}
Let the boundaries of $\Upsilon$ and $\Xi$ be well behaved, and let $f$ be a continuous function on $\cl(\Omega)$. Then
\begin{equation}\label{convergence}\Gamma_f=\lim_{l\rightarrow 0^+} l^dQ^l\boldsymbol{\Gamma}_{f,l}{P^l}^*.\end{equation}
(Also $\Theta_f=\lim_{l\rightarrow 0^+} l^dQ^l\boldsymbol{\Theta}_{f,l}{P^l}^*$).
\end{theorem}
\begin{proof}
We first establish that ${P^{l}}{P^{l}}^*$ converge to the identity operator $I$ in the \textit{SOT}-topology. Let $g\in L^2(\Upsilon)$ be arbitrary, pick any $\epsilon>0$ and let $\tilde g$ be a continuous function on $\cl(\Upsilon)$ with $\|g-\tilde{g}\|<\epsilon$. Then $$\|g-{P^{l}}{P^{l}}^*g\|\leq \|g-\tilde g\| +\|\tilde g-{P^{l}}{P^{l}}^*\tilde g\|+\|{P^{l}}{P^{l}}^*(\tilde{g}-g)\|.$$
Both the first and the last term are clearly $\leq \epsilon,$ whereas it is easy to see that the limit of the middle term as $l\rightarrow 0^+$ equals 0, since $\tilde{g}$ is continuous on $\cl(\Upsilon)$ and the boundary is well-behaved. Since $\epsilon$ was arbitrary we conclude that $\lim_{l\rightarrow 0^+}P^l{P^l}^*g=g$, as desired. The corresponding fact for $Q^l$ is of course then also true.

Now, since $\Gamma_f$ is compact by Corollary 2.4 in \cite{IEOT}, it follows by the above result and standard operator theory that $$\Gamma_f=\lim_{l\rightarrow 0^+}Q^l{Q^l}^*\Gamma_f P^l{P^l}^*,$$ and hence it suffices to show that $$0=\lim_{l\rightarrow 0^+}\|Q^l{Q^l}^*\Gamma_f P^l{P^l}^*-l^dQ^l\boldsymbol{\Gamma}_{f,l}{P^l}^*\|=\lim_{l\rightarrow 0^+}\|Q^l({Q^l}^*\Gamma_f P^l-l^d\boldsymbol{\Gamma}_{f,l}){P^l}^*\|.$$ Since $Q^l$ and ${P^l}^*$ are contractions, this follows if \begin{equation}\label{g5}\lim_{l\rightarrow 0^+}\|{Q^l}^*\Gamma_f P^l-l^d\boldsymbol{\Gamma}_{f,l}\|=0.\end{equation} By the Tietze extension theorem, we may suppose that $f$ is actually defined on $\R^n$ and has compact support there. In particular it will be equicontinuous. Now, to establish \eqref{g5}, let ${g}={g}_1+{g}_2\in\ell^2(\bb{\Upsilon}_l)$ be arbitrary, where $\supp {g}_1\subset \bb{\Upsilon}_l^{int}$ and $\supp {g}_2\subset \bb\Upsilon_l\setminus\bb{\Upsilon}_l^{int}$. By definition, $P^l {g}_2=0$ so ${Q^l}^*\Gamma_f P^l {g}_2=0$ whereas $$|l^d\boldsymbol{\Gamma}_{f,l}{g}_2(\bb{x})|\leq l^d \|f\|_{\infty}N_l(\Upsilon)^{1/2}\|{g}_2\|,$$ by the Cauchy-Schwartz inequality. Thus \begin{equation}\label{emil1}|({Q^l}^*\Gamma_f P^l-l^d\boldsymbol{\Gamma}_{f,l}){g}_2(\bb{x})|\leq l^d \|f\|_{\infty}N_l(\Upsilon)^{1/2}\|{g}_2\|.\end{equation} We now provide estimates for ${g}_1$. Given $\bb x\in\bb\Xi_l$ and $\bb y\in\bb\Upsilon_l$, set $$\tilde{f}({\bb{x}+\bb{y}})=\frac{1}{l^{2d}}\int_{|x-\bb{x}|_{\infty}<l/2}\int_{|y-\bb{y}|_{\infty}<l/2}f(x+y) ~dy~dx$$ and note that $$\tilde{f}({\bb{x}+\bb{y}})=\frac{1}{l^d}\langle {Q^{l}}^*\Gamma_f P^{l}\delta_{\boldsymbol{y}},\delta_{\boldsymbol{x}}\rangle$$ whenever $\bb x\in\bb\Xi_l^{int}$ and $\bb y\in\bb\Upsilon_l^{int}$. As in the proof of Theorem \ref{t1} it follows that
${Q^l}^*\Gamma_f P^l {g}_1(\bb x)=l^d\bb{\Gamma}_{\tilde{f},l} {g}_1(\bb x)$ for $\bb x\in\bb\Xi_l^{int}$. For such $\bb x$ we thus have \begin{equation}\label{emil2}|({Q^l}^*\Gamma_f P^l-l^d\boldsymbol{\Gamma}_{f,l}){g}_1(\bb{x})|=|l^d\bb{\Gamma}_{\tilde{f}-f,l} g_1(\bb x)|\leq l^d \|f-\tilde{f}\|_{\ell^2(\bb{\Omega}_l)}\|{g}_1\|\end{equation}
by Cauchy-Schwartz, and for $\bb x\in \bb \Xi\setminus\bb \Xi_{l}^{int}$ we get \begin{equation}\label{emil3}|({Q^l}^*\Gamma_f P^l-l^d\boldsymbol{\Gamma}_{f,l}){g}_1(\bb{x})|=|l^d\bb{\Gamma}_{f,l} g_1(\bb x)|\leq l^d \|f\|_{\infty}|\bb\Upsilon_l|^{1/2}\|{g}_1\|\end{equation} due to the definition of $Q^l$. Combining \eqref{emil1}-\eqref{emil3} we see that \begin{align*}&\|({Q^l}^*\Gamma_f P^l-l^d\boldsymbol{\Gamma}_{f,l}){g}\|\leq\|({Q^l}^*\Gamma_f P^l-l^d\boldsymbol{\Gamma}_{f,l}){g}_1\|+\|({Q^l}^*\Gamma_f P^l-l^d\boldsymbol{\Gamma}_{f,l}){g}_2\|\leq \\ \leq &|\bb{\Xi}_l^{int}|^{1/2}l^{d}\|f-\tilde{f}\|_{\bb{\Omega}_l}\|{g_1}\|+N_l(\Xi)^{1/2}l^d \|f\|_{\infty}|\bb\Upsilon_l|^{1/2}\|{g}_1\|+|\bb{\Xi}_l|^{1/2}l^d\|f\|_{\infty} N_l(\Upsilon)^{1/2}\|{g_2}\|. \end{align*}  Since $\Xi$ and $\Upsilon$ are bounded sets, $|\bb{\Xi}_l|$ and $|\bb \Upsilon_l|$ are bounded by some constant $C$ times $1/l^d$, and as $\| g_1\|\leq \| g\|$ and $\| g_2\|\leq \| g\|$, it follows that
$$\|({Q^l}^*\Gamma_f P^l-l^d\boldsymbol{\Gamma}_{f,l})\|\leq C^{1/2}l^{d/2}\|f-\tilde{f}\|_{\ell^2(\bb{\Omega}_l)}+ C^{1/2}N_l(\Xi)^{1/2}l^{d/2} \|f\|_{\infty}+ C^{1/2}l^{d/2}\|f\|_{\infty} N_l(\Upsilon)^{1/2}. $$
By Proposition \ref{t5} the last two terms go to 0 as $l$ goes to $0$. The same is true for the first term by noting that $l^{d/2}\|f-\tilde{f}\|_{\ell^2(\bb{\Omega}_l)}\leq \|f-\tilde{f}\|_{\ell^\infty(\bb{\Omega}_l)}l^{d/2}|\bb{\Omega}_l|^{1/2}$
and $$\lim_{l\rightarrow 0^+}\|f-\tilde{f}\|_{\ell^\infty(\bb{\Omega}_l)}=0,$$ which is an easy consequence of the equicontinuity of $f$. Thereby \eqref{g5} follows and the proof is complete.

\end{proof}
In particular, we have the following corollary. Note that the domains need not have well-behaved boundaries.
\begin{corollary}\label{c1} Let $\Upsilon$ and $\Xi$ be open, bounded and connected domains, and let $f$ be a continuous function on $\cl(\Omega)$. We then have
\begin{equation}\label{finiterank}\Rank\Gamma_f=\lim_{l\rightarrow 0^+}\Rank\boldsymbol{\Gamma}_{f,l}\end{equation}
(and $\Rank\Theta_f=\lim_{l\rightarrow 0^+}\Rank\boldsymbol{\Theta}_{f,l}$).
\end{corollary}
\begin{proof}
By Propositions 5.1 and 5.3 in \cite{IEOT}, the rank of $\Gamma_f$ is independent of $\Upsilon$ and $\Xi$. Combining this with Theorem \ref{t1}, it is easy to see that it suffices to verify the corollary for any open connected subsets of $\Upsilon$ and $\Xi$. We can thus assume that their boundaries are well-behaved. By Theorem \ref{t2} and standard operator theory we have
\begin{equation*}\Rank\Gamma_f\leq \liminf_{l\rightarrow 0^+} \Rank l^dQ^l\boldsymbol{\Gamma}_{f,l}{P^l}^*=\liminf_{l\rightarrow 0^+} \Rank Q^l\boldsymbol{\Gamma}_{f,l}{P^l}^*\leq \liminf_{l\rightarrow 0^+} \Rank\boldsymbol{\Gamma}_{f,l}.\end{equation*} On the other hand, Theorem \ref{t1} gives \begin{equation*} \limsup_{l\rightarrow 0^+}  \Rank\boldsymbol{\Gamma}_{f,l}\leq \Rank\Gamma_f.\end{equation*}
\end{proof}

\section{The multidimensional continuous Carath{\'e}odory-Fejer theorem}\label{seclast}

In the two final sections we investigate how the PSD-condition affects the theory. This condition only makes sense as long as $$\Xi=\Upsilon,$$ which we assume from now on. In this section we show that the natural counterpart of Carath{\'e}odory-Fejer's theorem hold for truncated convolution operators $\Theta_f$ on general domains $\Xi=\Upsilon$, and in the next we consider Fischer's theorem for truncated correlation operators.
\begin{theorem}\label{ty}
Suppose that $\Xi=\Upsilon$ is open bounded and connected, $\Omega=\Xi-\Upsilon$, and $f\in L^2(\Omega)$. Then the operator $\Theta_f$ is PSD and has finite rank $K$ if and only if there exists ${\xi}_1,\ldots,{\xi}_K\in\R^d$ and $c_1,\ldots,c_K>0$ such that
\begin{equation}\label{gggg}f=\sum_{k=1}^K c_k e^{i {\xi}_k\cdot {x}}.\end{equation}
\end{theorem}
\begin{proof}
Suppose first that $\Theta_f$ is PSD and has finite rank $K$. By Theorem 4.4 in \cite{IEOT}, $f$ is an exponential polynomial (i.e. can be written as \eqref{rep2}). By uniqueness of analytic continuation, it suffices to prove the result for $\Xi=\Upsilon$ are neighborhoods of some fixed point ${x}_0$. By a translation, it is easy to see that we may assume that ${x}_0=0$. Let $l$ assume values $2^{-j}$, $j\in\N$. For $j$ large enough, (beyond $J$ say), the operator $\bb{\Gamma}_{f,2^{-j}}$ has rank $K$ (Corollary \ref{c1}) and Theorem \ref{CFMD} applies (upon dilation of the grids). We conclude that for $j>J$ the representation \eqref{gggg} holds (on $\bb\Omega_{2^{-j}}$) but the ${\xi}_k$'s may depend on $j$. However, since each grid $\bb{\Omega}_{2^{-j-1}}$ is a refinement of $\bb{\Omega}_{2^{-j}}$, Proposition \ref{linindep} guarantees that this dependence on $j$ may only affect the ordering, not the actual values of the set of $\xi_k$'s used in \eqref{gggg}. We can thus choose the order at each stage so that it does not depend on $j$. Since $f$ is an exponential polynomial, it is continuous, so taking the limit $j\rightarrow\infty$ easily yields that \eqref{gggg} holds when $\bb{x}$ is a continuous variable as well.

Conversely, suppose that $f$ is of the form \eqref{gggg}. Then $\Theta_f$ has rank $K$ by Proposition 4.1 in \cite{IEOT} (see also the remarks at the end of Section \ref{contcase}). The PSD condition follows by the continuous analogue of \eqref{apa1}.
\end{proof}

\section{PSD Truncated Correlation Operators}\label{secPSDGAmma}

\begin{theorem}\label{fisher}
Suppose that $\Xi=\Upsilon$ is open bounded and connected, $\Omega=\Xi+\Upsilon$, and $f\in L^2(\Omega)$. The operator $\Gamma_f$ is PSD and has finite rank $K$ if and only if there exists ${\xi}_1,\ldots,{\xi}_{K}\in\R^d$ and $c_1,\ldots,c_{K}>0$ such that
\begin{equation}\label{gggghytf}f=\sum_{k=1}^{K} c_k e^{{\xi}_k\cdot {x}}.\end{equation}
\end{theorem}

We remark that the continuous version above differs significantly from the discrete case, even in one dimension, since the sequence $(\lambda^n)_{n=0}^{2N}$ generate a PSD Hankel matrix for all $\lambda\in\R$, whereas the base $e^{\xi_k}$ is positive in \eqref{gggghytf}.

\begin{proof}
Surprisingly, the proof is rather different than that of Theorem \ref{ty}. First suppose that $\Gamma_f$ is PSD and has finite rank $K$. Then $f$ is an exponential polynomial, i.e. has a representation \eqref{rep2}, by Theorem 4.4 in \cite{IEOT}. Suppose that there are non-constant polynomial factors in the representation \eqref{rep2}, say $p_1$. Let $N$ be the maximum degree of all polynomials $\{p_j\}_{j=1}^J$. Pick a closed subset $\tilde\Xi\subset\Xi$ and $r>0$ such that $\dist(\tilde\Xi,\R^d\setminus\Xi)>2r$. Pick a continuous real valued function $g\in L^2(\R^d)$ with support in $\tilde\Xi$ that is orthogonal to the monomial exponentials $$\{ x^{ \alpha} e^{ \zeta_j\cdot  x}\}_{| \alpha|\leq N, 1\leq j\leq J}\setminus\{e^{ \zeta_1\cdot x}\}$$
(where ${\alpha}\in \N^d$ and we use standard multi-index notation), but satisfies $\langle g,e^{ \zeta_1\cdot x}\rangle=1$, (that such a function exists is standard, see e.g. Proposition 3.1 in \cite{IEOT}). A short calculation shows that \begin{equation}\label{lok}\langle \Gamma_f g(\cdot- z),g(\cdot- w)\rangle=p_1( z+ w)e^{ \zeta_1\cdot( z+ w)}\end{equation}
whenever $| z|,| w|<r$. Since $p_1$ is non-constant, there exists a unit length $ \nu\in\R^d$ such that $q(t)=p_1(r \nu t)$ is a non-constant polynomial in $t$. Set $\zeta=r \zeta_1\cdot \nu$.
Consider the operator $A:L^2([0,1])\rightarrow L^2(\Xi)$ defined via $$A(\phi)= \int_0^1 \phi(t) g( x- {r}\nu t).$$
Clearly $A^*\Gamma_f A$ is selfadjoint. It follows by \eqref{lok} and Fubini's theorem that
\begin{align*}&\langle A^*\Gamma_f A(\phi),\psi \rangle=\int_0^1\int_0^1 p_1(r\nu t+r \nu s)e^{\zeta_1\cdot(r\nu t+r \nu s)}\phi(t)dt\overline{\psi(s)}ds=\\&\int_0^1\int_0^1 q(t+s) e^{\zeta(t+s)}\phi(t)dt\overline{\psi(s)}ds.\end{align*}
With $h(t)=q(t) e^{\zeta t}$, it follows that the operator $\Gamma_h:L^2([0,1])\rightarrow L^2([0,1])$ is self-adjoint.
Either by repeating arguments form Section \ref{multikronecker}, or by standard results from integral operator theory, it is easy to see that $h(t+s)=\overline{h(s+t)}$, i.e. $h$ is real valued. This clearly implies that $\zeta\in\R$. Now consider the operator $B:L^2([0,1])\rightarrow L^2([0,1])$ defined by $B(g)(t)=e^{-\zeta t}g(t)$. As before we see that $B^*\Gamma_hB=\Gamma_q$ is PSD. Given $0<\epsilon<1/2$, define $C_{\epsilon}:L^2([0,1/2])\rightarrow L^2([0,1])$ by $C_{\epsilon}(g)(t)=\frac{g(t-\epsilon)-g(t)}{\epsilon}$, (where we identify functions on $[0,1/2]$ with functions on $\R$ that are identically zero outside the interval). It is easy to see that $$C_{\epsilon}^* \Gamma_q\C_{\epsilon}=\Gamma_{{\epsilon^{-2}}{(q(\cdot+2\epsilon)-2q(\cdot+\epsilon)+q(\cdot))}},$$
which means that also this truncated correlation operator is PSD. Since $q$ is a polynomial, it is easy to see that
$(q(\cdot+2\epsilon)-2q(\cdot+\epsilon)+q(\cdot))/\epsilon^2$ converges uniformly on compacts to $q''$. By simple estimates based on the Cauchy-Schwartz inequality (see e.g. Proposition 2.1 in \cite{IEOT}), it then follows that the corresponding sequence of operators converges to $\Gamma_{q''}$ (acting on $L^2([0,1/2])$), which therefore is PSD. Continuing in this way, we see that we can assume that $q$ is of degree 1 or 2, where $\Gamma_q$ acts on an interval $[0,3l]$ where $3l$ is a power of $1/2$. We first assume that the degree is 2, and parameterize $q(t)=a+b(t/l)+c(t/l)^2$. Performing the differentiation trick once more, we see that $\Gamma_c$ is PSD on some smaller interval, which clearly means that $c>0$. Now pick $g\in L^2([0,l])$ such that $\langle g,1\rangle=1$, $\langle g,t\rangle=0$, $\langle g,t^2\rangle=0$, and consider $D:\C^3\rightarrow L^2([0,3l])$ defined by $$D((c_0,c_1,c_2))=c_0 g(\cdot)+c_1 g(\cdot-l)+c_2 g(\cdot-2l).$$ By \eqref{lok}, the matrix representation of $D^*\Gamma_q D$ is $$M=\left(
                                        \begin{array}{ccc}
                                          q(0) & q(l) & q(2l) \\
                                          q(l) & q(2l) & q(3l) \\
                                          q(2l) & q(3l) & q(4l) \\
                                        \end{array}
                                      \right)=\left(
                                        \begin{array}{ccc}
                                          a & a+b+c & a+2b+4c \\
                                          a+b+c & a+2b+4c & a+3b+9c \\
                                          a+2b+4c & a+3b+9c & a+4b+16c \\
                                        \end{array}
                                      \right),
$$
which then is PSD. However, a (not so) short calculation shows that the determinant of $M$ equals $-8c^3$ which is a contradiction, since it is less than 0 (recall that $c>0$). We now consider the case of degree 1, i.e. $c=0$ and $b\neq 0$. As above we deduce that the matrix
$$M=\left(
                                        \begin{array}{ccc}
                                          q(0) & q(l) \\
                                          q(l) & q(2l) \\
                                        \end{array}
                                      \right)=\left(
                                        \begin{array}{ccc}
                                          a & a+b+c \\
                                          a+b+c & a+2b+4c \\
                                        \end{array}
                                      \right),
$$
has to be PSD, which contradicts the fact that its determinant is $-b^2$.

By this we finally conclude that there can be no polynomial factors in the representation \eqref{rep2}. By the continuous version of Proposition \ref{rankprop} (see Proposition 4.1 in \cite{IEOT}), we conclude that $f$ is of the form \eqref{rep3}, i.e. $f=\sum_{k=1}^Kc_ke^{\zeta_k\cdot x}$. From here the proof is easy. Repeating the first steps, we conclude that $\zeta_k\cdot\nu\in\R$ for all $ \nu\in\R^d$, by which we conclude that $\zeta_k$ are real valued. We therefore call them $ \xi_k$ henceforth. With this at hand we obviously have
\begin{equation}\label{apa19}\langle \boldsymbol{\Gamma}_{f}(g),g\rangle=\sum_{k=1}^Kc_k|\langle g,e^{ {\xi}_k\cdot {x}}\rangle|^2\end{equation} for all $g\in L^2(\Xi)$,
whereby we conclude that $c_k>0$.

For the converse part of the statement, let $f$ be of the form \eqref{gggghytf}. That $\Gamma_f$ has rank $K$ has already been argued (Proposition 4.1 in \cite{IEOT}) and that $\Gamma_f$ is PSD follows by \eqref{apa19}. The proof is complete.
\end{proof}

\section{Unbounded domains}\label{secUB}

For completeness, we formulate the results form the previous two sections for unbounded domains. $\Gamma_f$ is defined precisely as before, i.e. via the formula \eqref{TCOdd}, except that we now have to assume that $f(x+\cdot)$ is in $L^2(\Upsilon)$ for every $x\in\Xi$ and vice versa, (see definition 1.1 in \cite{IEOT}). Obviously, analogous definitions/restrictions apply to $\Theta_f$ as well. The main difficulty with unbounded domains is that exponential polynomials then can give rise to unbounded operators. Following \cite{IEOT}, we address this by assuming that $\Omega$ is convex and we let $\Delta_{\Omega}$\footnote{It was called $\Theta_{\Omega}$ in \cite{IEOT}, see Section 9.} denote the set of directions ${\vartheta}\in\R^d$ such that the orthogonal projection of $\Omega$ on the half line $[0,\infty)\cdot{\vartheta}$ is a bounded set, and we let $\interior(\Delta_{\Omega})$ denote its interior.

\begin{theorem}\label{t6}
Let $\Xi=\Upsilon\subset\R^d$ be convex domains and set $\Omega=\Xi+\Upsilon$. Then $\Gamma_f$ is bounded, PSD and has finite rank if and only if $f$ is of the form \eqref{gggghytf} and ${\xi}_k\in \interior(\Delta_{\Omega})$ for all $k.$
\end{theorem}
\begin{proof}
This follows by straightforward modifications of the proofs in Section 9 of \cite{IEOT}, so we satisfy with outlining the details. The ``if'' direction is easy so we focus on the ``only if''. We restrict the operator $\Gamma_f$ to functions living on a subset (see Theorem 9.1 \cite{IEOT}) to obtain a new operator to which Theorem \ref{fisher} above applies. From this we deduce that $f$ locally has the form \eqref{gggghytf}. That this formula then holds globally is an immediate consequence of uniqueness of real analytic continuation, combined with the observation that $\Omega$ is connected. Finally, the restriction on the $\xi_k$'s is immediate by Theorem 9.3 in \cite{IEOT}.
\end{proof}

The corresponding situation for truncated convolution operators is quite different. We first note that $\Theta_f:L^2(\Upsilon)\rightarrow L^2(\Xi)$ is bounded if and only if $\Gamma_f:L^2(-\Upsilon)\rightarrow L^2(\Xi)$ is bounded, as mentioned in Section \ref{contcase} and further elaborated on around formula (1.2) in \cite{IEOT}. With this, we immediately obtain the following theorem, which was left implicit in \cite{IEOT}.

\begin{theorem}\label{t7}
Let $\Xi,\Upsilon\subset\R^d$ be convex domains and set $\Omega=\Xi-\Upsilon$. Then $\Theta_f$ is bounded and has finite rank if and only if $f$ is an exponential polynomial (i.e. $
f({x})=\sum_{j=1}^J p_j({x}) e^{{\zeta}_j\cdot {x}}$) and $\re{\zeta}_j\in \interior(\Delta_{\Omega})$ for all $j.$
\end{theorem}

However, if now again we let $\Xi=\Upsilon$ and we additionally impose PSD, the proof of Theorem \ref{t6} combined with Theorem \ref{ty} shows that $ \zeta_j=i\xi_j$ for some $\xi\in\R^d$. However, Theorem \ref{t7} then forces $0=\re \zeta_j\in \interior (\Delta_{\Omega})$, which can only happen if $\Delta_{\Omega}=\R^d$, since it is a cone. This in turn is equivalent to $\Omega$ being bounded, so we conclude that

\begin{theorem}\label{t8}
Let $\Xi=\Upsilon\subset\R^d$ be convex unbounded domains and set $\Omega=\Xi-\Upsilon$. Then $\Theta_f$ is bounded and PSD if and only if $f\equiv 0$.
\end{theorem}

%

\section{Conclusions} Multidimensional versions of the Kronecker, Carath{\'e}odory-Fejer and Fischer theorems are discussed and proven in discrete and continuous settings. The former relates the rank of general domain Hankel and Toeplitz type matrices and operators to the number of exponential polynomials needed for the corresponding generating functions/symbols. The latter two include the condition that the operators be positive semi-definite. The multi-dimensional versions of the Carath{\'e}odory-Fejer theorem behave as expected, while the multi-dimensional versions of the Kronecker theorem generically yield more complicated representations, which are clearer in the continuous setting. Fischer's theorem also exhibits a simpler structure in the continuous case than in the discrete. We also show that the discrete case approximates the continuous, given sufficient sampling.

\section*{Acknowledgement} This research was partially supported by the Swedish Research Council, Grant No 2011-5589.

\bibliographystyle{plain}
\bibliography{MCref}

\end{document}